\documentclass{amsart}
\usepackage{amssymb}
\usepackage[all, cmtip]{xy}
\usepackage[hyphens]{url} 
\usepackage[utf8]{inputenc}
\usepackage{amsmath}
\usepackage{amsfonts}
\usepackage{amssymb}
\usepackage{enumerate}
\usepackage{amsthm}
\usepackage{centernot}
\usepackage{tikz, pgfplots, xcolor, fancyhdr,multicol}
\usepackage[colorlinks,citecolor=blue,urlcolor=blue,bookmarks=false,hypertexnames=true]{hyperref} 

\def\s{{\sigma}}
\def\o{{\omega}}
\def\N{{\mathbb N}}
\def\R{{\mathbb R}}
\def\V{{\mathcal{V}}}
\def\P{{\mathcal{P}}}

\def\E{{\mathcal{E}}}
\def\A{{\mathcal{A}}}
\usepackage{graphicx} 
\graphicspath{{figures/}{C:\Users\grigo\Desktop\weighted inequalities/}}
\newif\ifgrading

\gradingfalse

\newcommand{\RC}[1]{\ifgrading {{\color{blue} #1 }} \fi}

\theoremstyle{plain}
 \newtheorem{thm}{Theorem}
 
 \newtheorem{lem}{Lemma}
 \newtheorem{cor}{Corollary}
 \newtheorem{question}{Question}
\theoremstyle{definition}

\theoremstyle{remark}
 \newtheorem{rem}{Remark}[section]
 
 \numberwithin{equation}{section}

\renewcommand{\leq}{\leqslant}
\renewcommand{\geq}{\geqslant}

\DeclareMathOperator{\dist}{dist}

\title[Necessary/sufficient conditions in weighted theory]{The necessary/sufficient conditions in weighted theory}



\author[Grigoriadis]{\bfseries  Christos Grigoriadis} 

\address{ 
Department of Mathematics \\ 
Michigan State University   \\ 
East Lansing\\
US}
\email{grigori4@math.msu.edu}

\pgfplotsset{compat=1.14}

\begin{document}

\maketitle

\begin{abstract}
 We provide  an essentially complete dictionary of all implications among the basic and fundamental conditions in weighted theory such as the doubling, one weight $A_p(w)$, $A_\infty$ and $C_p$ conditions as well as the two weight $\A_p(\o,\s)$ and the ``buffer" Energy and Pivotal conditions. The most notable implication is that in the case of $A_\infty$ weights the two weight $\A_p$ condition implies the $p-$Pivotal condition hence giving an elegant and short proof of the known NTV-conjecture with $p=2$ for $A_\infty$ weights in terms of existing T1 theory. We also provide a quite technical construction inspired by \cite{GaKS} proving that we can have  doubling weights satisfying the $C_p$ condition which are not in $A_\infty$.
\end{abstract}

\section{Introduction}

Given two locally finite positive Borel measures $\o,\s$ in $\R^n$, the two weight problem for an operator $T$ is to characterize $\o,\s$ so that  

\begin{equation}
\label{main}
||T(fd\s)||_{L^p(\o)} \lesssim ||f||_{L^p(\s)}, \quad \forall f \in L^p(\s),\;\; p>1.
\end{equation}

\subsection{One weight theory}

\eqref{main} is a generalization of the one weight inequality for the Hilbert transform  where $T=H$, $d\o(x)=w(x)dx$, $d\s(x)=w(x)^{1-p\prime}dx$ and $f\in L^p(w)$
\begin{equation}\label{Hilbert one weight}
||Hf||_{L^p(\o)} \lesssim ||f||_{L^p(\o)}
\end{equation}
which was shown by Hunt, Muckenhoupt and Wheeden \cite{HMW} to be equivalent to the finiteness of the  Muckenhoupt one weight $A_p$ condition, namely $\o$ has to be absolutely continuous to Lebesgue measure $d\o =w(x)dx$ and
\begin{equation}
\label{one weight Ap}
A_p(w)=\underset{I}{\sup}\frac{1}{|I|}\int_I w(x)dx\left( \frac{1}{|I|}\int_Iw(x)^{\frac{1}{1-p}} dx\right)^{p-1}\leq C<\infty
\end{equation}
where the supremum is taken uniformly over all cubes in $\R^n$. There has been a huge amount of work in harmonic analysis and boundary value problems around the $A_p$ condition, check Stein \cite{St}, Duoandikoetxea \cite{DJ}, Garnett \cite{Ga} and references there.

Coifman and Fefferman in \cite{CoFe} proved \eqref{Hilbert one weight} using the following inequality, which holds for any $w \in A_\infty=\bigcup_{p\geq 1}A_p$,
\begin{equation}\label{Hilbert maximal}
\int_{\R^n}|Tf(x)|^pw(x)dx\leq C\int_{\R^n}|Mf(x)|^pw(x)dx
\end{equation}
where $T$ is any singular integral operator and $f$ is bounded and compactly supported. We can extend to any locally integrable $f$ for which the right hand side is finite (since otherwise there is nothing to prove) using the dominated convergence theorem.

Muckenhoupt in \cite{Mu} proved, for $n=1$, that a more general class of weights than the $A_p$ weights, namely the $C_p$ weights (see \eqref{Cp condition}), are necessary for \eqref{Hilbert maximal} to hold. This was generalized in higher dimensions by Sawyer in \cite{Saw3}. Sawyer in \cite{Saw3} also shows that the $C_q$ condition for $q>p$ is sufficient for \eqref{Hilbert maximal} to hold. It is still unknown if the $C_p$ condition is sufficient for \eqref{Hilbert maximal} to hold.

We say the measure $\o$ satisfies the $C_p$ condition, $1<p<\infty$ if it is absolutely continuous to the Lebesgue measure, i.e. $d\o=w(x)dx$, and there exist $C,\epsilon>0$ such that
\begin{equation}\label{Cp condition}
\frac{|E|_w}{\int_{\R^n}|M\mathbf{1}_I(x)|^pw(x)dx}\leq C\left(\frac{|E|}{|I|}\right)^\epsilon, \text{for } E \text{ compact subset of  }I \text{ cube}
\end{equation}
with  $\int_{\R^n} \left(M\mathbf{1}_I\left(x\right)\right)^pw(x)dx<\infty$, where $|E|_w=\int_Ew(x)dx$. Here $M$ denotes the classical Hardy-Littlewood maximal operator. We will call  $w(x)$ a $C_p$ weight.

We prove that $C_p$ weights is a strictly larger class than the $A_\infty$ weights. We actually show that there exist even doubling weights (see
\eqref{doubling}) that are also $C_p$ weights that are not in $A_\infty$. Check the diagram at the end of the introduction.

\begin{thm}\label{Cp theorem}
($C_p\cap\mathcal{D}\nRightarrow A_\infty$) There exist a weight $w$ that is doubling and satisfies the $C_p$ condition but $w$ is not an $A_\infty$ weight.
\end{thm}
The weight $w$ used in theorem \ref{Cp theorem}. has a doubling constant $C_w\gtrsim 3^{np}$. We show that this is sharp, i.e. if the doubling constant $C_w$ of the weight $w$ does not satisfy $C_w \geq 3^{np}$ then the $C_p$ condition is equivalent to $A_\infty$.

\begin{thm}\label{Cp theorem small doubling}
($C_p$+small doubling $\Rightarrow \A_\infty$) Let w be a doubling $C_p$ weight with doubling constant $C_w< 3^{np}$ in $\R^n$. Then $w \in A_\infty$.
\end{thm}
\subsection{Two weight theory.}

The generalization of the one weight $A_p$ condition was naturally modified to the two weight problem by:
\begin{equation}
\label{classical Ap}
\A_p(\o,\s)=\underset{I}{\sup}\left(\frac{\o(I)}{|I|}\right)^\frac{1}{p}\left(\frac{\s(I)}{|I|}\right)^\frac{1}{p\prime}<\infty
\end{equation}
where the supremum is taken over all cubes in $\R^n$ and the weight $w$ gives its place to two positive locally finite Borel measures. Notice that by setting $d\o=w(x)^{\frac{1}{1-p}}dx$, $d\s=w(x)dx$
we retrieve the one weight $A_p$ condition \eqref{one weight Ap}.

The two weight problem could have applications in a number of problems connected to higher dimensional analogs of the Hilbert transform. For example,  questions regarding subspaces of the Hardy space invariant under the inverse shift operator (see \cite{Vol}, \cite{NaVo}),questions concerning orthogonal polynomials (see \cite{VoYu}, \cite{PeVoYu}, \cite{PeVoYu1}) and some questions in quasiconformal theory for example the conjecture of Iwaniec and Martin (see \cite{IwMa}) or higher dimensional analogues of the Astala conjecture (see \cite{LSU1}).  

The classical $\A_p$ condition \eqref{classical Ap} is necessary for \eqref{main} to hold but is no longer sufficient, which is an indication that makes two weight theory much more complicated. F. Nazarov in \cite{Naz} has shown that even the strengthened $\A_p(\o,\s)$ conditions with one or two tails of Nazarov, Treil and Volberg
\begin{equation}
\label{one tailed}
\A_p^{t_1}(\o,\s)=\underset{I}{\sup} \left(\frac{\o(I)}{|I|}\right)^\frac{1}{p}\left(P(I,\s)\right)^\frac{1}{p\prime}<\infty
\end{equation}
\begin{equation}
\label{two tailed}
\A_p^{t_2}(\o,\s)=\underset{I}{\sup} \left(P(I,\o)\right)^\frac{1}{p}\left(P(I,\s)\right)^\frac{1}{p\prime}<\infty
\end{equation}
where 
\begin{equation}\label{poisson integral}
P(I,\o)\equiv \int_\R\left(\frac{|I|^\frac{1}{n}}{(|I|^\frac{1}{n}+\dist(x,I))^{2}}\right)^n\o(dx)
\end{equation}
along with their duals $\A^{t_1,*}_p(\o,\s),\A^{t_2,*}_p(\o,\s)$, where the roles of $\s$ and $\o$ are interchanged, are no longer sufficient for \eqref{main} to hold. 

When the operator $T$ in \eqref{main} is a fractional operator such as the Cauchy transform or the fractional Riesz transforms then the fractional analogs of \eqref{classical Ap}, \eqref{one tailed}, \eqref{two tailed} are used
\begin{equation}
\label{fractional Ap}
\A^\alpha_p(\o,\s)=\underset{I}{\sup}\left(\frac{\o(I)}{|I|^{1-\frac{\alpha}{n}}}\right)^\frac{1}{p}\left(\frac{\s(I)}{|I|^{1-\frac{\alpha}{n}}}\right)^\frac{1}{p\prime}<\infty
\end{equation}
\begin{equation}
\label{fractional one tailed}
\A_p^{t_1,\alpha}(\o,\s)=\underset{I}{\sup} 
\left(\frac{\o(I)}{|I|^{1-\frac{\alpha}{n}}}\right)^\frac{1}{p}\left(\P^\alpha(I,\s)\right)^\frac{1}{p\prime}<\infty
\end{equation}
\begin{equation}
\label{fractional two tailed}
\A_p^{t_2,\alpha}(\o,\s)=\underset{I}{\sup} \left(\P^\alpha(I,\o)\right)^\frac{1}{p}\left(\P^\alpha(I,\s)\right)^\frac{1}{p\prime}<\infty
\end{equation}
where $\P^\alpha$ is the \textit{reproducing} Poisson integral and is given by 

$$\P^\alpha(I,\o)\equiv \int_{\R^n}\left(\frac{|I|^\frac{1}{n}}{(|I|^\frac{1}{n}+\dist(x,I))^{2}}\right)^{n-\alpha}\o(dx)$$
The \textit{standard} Poisson integral, is given by
$$P^\alpha(I,\o)\equiv \int_{\R^n}\frac{|I|^\frac{1}{n}}{(|I|^\frac{1}{n}+\dist(x,I))^{n+1-\alpha}}\o(dx)$$
and is used for the definition of the fractional ``buffer" conditions. The two Poisson integrals agree for $n=1$, $\alpha=0$.
We refer the reader to \cite{SSU4} for more details. All the results that we are proving here for the $\A_p$ conditions hold for their fractional analogs without any modification in the proofs.

We show that the classical $\A_p$ condition is weaker than the tailed conditions, but the two tailed $\A_p$ condition holding is equivalent to both one tailed $\A_p$ conditions holding.

\begin{thm}\label{non doubling Ap examples} We have the following implications:
\begin{enumerate}
    \item ($\A_p \nRightarrow \A_p^{t_1}\cap \A_p^{t_1,*}$) The two weight classical $\A_p$ condition does not imply the one tailed $\A_p$ conditions.
    
    \item ($\A_p^{t_1}\nRightarrow \A_p^{t_2}$) The one tailed $\A_p^{t_1}$ condition does not imply the two tailed $\A_p^{t_2}$ condition.
    
    \item ($\A_p^{t_1}\cap \A_p^{t_1,*} \Leftrightarrow \A_p^{t_2}$) The two tailed $\A_p^{t_2}$ condition holding is equivalent to both one tailed $\A_p^{t_1},\A_p^{t_1,*}$ conditions holding. 
\end{enumerate}
\end{thm}
The measures that we use for the proof of theorem \ref{non doubling Ap examples}. are non doubling and we show that this is the only case. All doubling measures are reverse doubling (see lemma \ref{reverse doubling implies doubling}.). So the previous sentence is justified by the following theorem:
\begin{thm}\label{Ap doubling equivalence theorem}($\o,\s\in \mathcal{D}$ , $\A_p \Rightarrow \A_p^{t_1}\Rightarrow \A_p^{t_2}$)
If $\o,\s$ are reverse doubling measures, then the classical two weight classical $\A_p$ implies the  tailed $\A_p$ conditions. 
\end{thm}

\subsection{The testing conditions.} Since the two weight $\A_p$ conditions are not sufficient for \eqref{main} to hold, some other necessary conditions are required, namely the \textbf{1}-testing conditions
\begin{eqnarray}
\label{test}
||T(\textbf{1}_Id\s)||_{L^p(\o)} &\leq& \mathfrak{T}^p|I|_\s \\
||T^*(\textbf{1}_Id\o)||_{L^p(\s)} &\leq& (\mathfrak{T}^*)^p|I|_\o\nonumber
\end{eqnarray}
where $I$ runs over all cubes and $\mathfrak{T}, \mathfrak{T}^*$ are the best constants so that  \eqref{test} holds.

These conditions alone are  trivially not  sufficient for \eqref{main} to hold since as pointed out in \cite{NiTr} for example, the second Riesz transform $R_2$ of any measure supported on the real line is the zero element in $L^p(\o)$ for any measure $\o$ carried by the upper half plane. On the other hand, such a pair of measures need not satisfy the Muckenhoupt conditions, which are necessary for \eqref{main} to hold.

The famous Nazarov-Treil-Volberg conjecture (NTV conjecture), states that $\A_p(\o,\s)$ and testing conditions are necessary and sufficient for \eqref{main} to hold.

\subsection{The ``buffer" Pivotal and Energy conditions.} Nazarov, Treil and Volberg in a series of very clever papers assumed the pivotal condition, for $p=2$, and proved \eqref{main} (see \cite{NTV1},\cite{NTV2},\cite{Vol}).

The Pivotal condition $\V$ is given by
\begin{equation}
\label{pivotal}
\V(\o,\s)^p=\underset{I_0=\cup I_r}{\sup}\frac{1}{\s(I_0)}\displaystyle\sum_{r \geq 1}\o(I_r)P(I_r,1_{I_0}\s)^p<\infty
\end{equation}
where the supremum is taken over all possible decompositions of $I_0$ in disjoint cubes $\{I_r\}_{r \in \N}$ and all cubes $I_0$ such that $\s(I_0) \neq 0$, and its dual $\V^*$ where $\s$ and $\o$ are interchanged.

Lacey, Sawyer and Uriarte-Tuero in \cite{LSU} proved, again for $p=2$, that \eqref{main} for the Hilbert transform implies  the weaker Energy condition $\E$ 
\begin{equation}
\label{energy}
\E(\o,\s)^p=\underset{I_0=\cup I_r}{\sup}\frac{1}{\s(I_0)}\displaystyle\sum_{r \geq 1}\o(I_r)E(I_r,\o)^2P(I_r,1_{I_0}\s)^p<\infty
\end{equation}
where the supremum is taken over all possible decompositions of $I_0$ in disjoint cubes $\{I_r\}_{r \in \N}$ and all cubes $I_0$ such that $\s(I_0) \neq 0$, where
\begin{equation}\label{energy gain}
E(I,\o)^2 \equiv \frac{1}{2}\mathbb{E}_I^{\o (dx)}\mathbb{E}_I^{\o (dx')}\frac{(x-x')^2}{|I|^2}
\end{equation}
and its dual $\E^*$ where $\s$ and $\o$ are interchanged.

In the same paper, Lacey, Sawyer and Uriarte-Tuero proved that a hybrid of the Pivotal and Energy conditions was sufficient but not necessary in the two weight inequality for the Hilbert transform.

Both the energy and the pivotal conditions, sometimes referred to as ``buffer conditions", are used to approximate certain forms that appear in the proofs of almost all two weight inequalities. The NTV conjecture states that we can prove $\eqref{main}$ without assuming them.

It is true though that if both $\o,\s$ are individually $\A_\infty$ weights, the classical $\A_p(\o,\s)$ condition implies the Pivotal condition providing a short and elegant proof of the NTV-conjecture for $\A_\infty$ weights assuming existing $T1$ theory. Earlier, Sawyer in \cite{Saw2} gave a proof using different methods for the case of smooth kernels. 
\begin{thm}\label{T1 theorem} ($T1$ theorem for $\A_\infty$ weights)
Assume $\o,\s$ are in $A_\infty$, $T$ is an $\alpha$-fractional singular integral  and we have the $T1$ testing and the fractional $\A_2^\alpha(\o,\s)$ conditions to hold, along with their duals. Then, $T$ is bounded on $L^2(\R^n)$.
\end{thm}

\subsection{The relationship between the two weight $\A_p$ and ``buffer" conditions. }It is shown in \cite{LSU} that we can have a pair of measures satisfying the tailed $\A_2$ conditions \eqref{one tailed}, \eqref{two tailed} but failing to satisfy the Pivotal condition \eqref{pivotal}, hence proving the implication $\A_2^{t_2}\nRightarrow \V^2$.

We show here that the Pivotal condition \eqref{pivotal} does not imply the tailed $\A_2$ conditions \eqref{one tailed}, \eqref{two tailed}. 
\begin{thm}\label{non doubling pivotal example}($\V^p \nRightarrow \A_p^{t_1}$) Let $1<p\leq 2$.
The Pivotal condition $\V^p$ does not imply the one tailed $\A_p$ condition $\A_p^{t_1}$. 
\end{thm}
\begin{rem}\label{Pivotal implies Energy}
It is immediate from \eqref{energy gain} that the  Energy condition \eqref{energy} is dominated by the Pivotal condition \eqref{pivotal} hence we immediately get the following important corollary.
\end{rem}

\begin{cor}\label{energy corollary}($\E\nRightarrow \A_2^{t_1}$)
Let $1<p\leq 2$.
The Energy condition $\E$ does not imply the one tailed $\A_p$ condition $\A_p^{t_1}$. \qed
\end{cor}

\subsection{Organization of the paper}

In section 4 we prove theorems \ref{Cp theorem} and \ref{Cp theorem small doubling}. In section 5.1 we prove theorem \ref{non doubling Ap examples}. and in section 5.2 we prove theorem \ref{Ap doubling equivalence theorem}. We prove the $T1$ theorem for $\A_\infty$ weights, theorem \ref{T1 theorem}, in section 5.3, using the Sawyer testing condition (see \eqref{Sawyer testing} and theorem \ref{Sawyer testing theorem}). In section 5.4 we prove theorem \ref{non doubling pivotal example} and give a partial answer to question \ref{doubling measures question} in theorem \ref{Ap and small doubling corollary}. Check the graph and the lattices in sections 2 and 3 for a summary in the $T1$ theory and the theorems presented in this paper.

\subsection{Known cases of the NTV conjecture.}While the general case of the NTV 
conjecture in $\R^n$ is still not completely understood, several important special 
cases have been completely solved. 

First, in the two part paper by Lacey, Sawyer, 
Shen and Uriarte-Tuero \cite{LSSU} and Lacey \cite{Lac} proved the NTV conjecture, 
namely that $\A_p(\o,\s)$ and testing conditions are necessary and sufficient for 
\eqref{main} to hold, assuming also that the measures $\s$ and $\o$ had no common 
point masses, for the Hilbert Transform. Hyt{\"o}nen \cite{Hyt} with his new offset 
version of $\A_2$
\begin{equation}
\label{offset A2}
\A_2^{\text{offset}}(\o,\s)=\underset{I}{\sup}\frac{\o(I)}{|I|}\int_{\R^n\backslash I}\left(\frac{|I|^\frac{1}{n}}{(|I|^\frac{1}{n}+\dist(x,I))^2}\right)^n\s(dx) <\infty
\end{equation}
removed the restriction of common point masses on $\s, \o$. An alternate approach using ``punctured" versions of $\A_2$ appears in \cite{SSU5}.

Other important cases include first Sawyer, Shen, Uriarte-Tuero \cite{SSU4} for $\alpha$-fractional singular integrals, Lacey-Wick \cite{LW} for the Riesz transforms, Lacey, Sawyer, Shen, Uriarte-Tuero and Wick \cite{LSSUW} for the Cauchy transform and Sawyer, Shen, Uriarte-Tuero \cite{SSU2} for the Riesz tranform when a measure is supported on a curve in $\R^n$ and recently \cite{Saw2} for general Calderon-Zygmund operators and doubling measures that also satisfy the fractional $\A_\infty^\alpha$ condition, check \eqref{A alpha infinity}. The NTV conjecture is yet to be proven for a general operator $T$.

\textbf{Acknowledgements:} I would like to thank my advisors Eric Sawyer and Ignacio Uriarte-Tuero for introducing me the area, presenting the problem to me and providing suggestions for its progress.

\section{Lattices}
\begin{center} 
\textbf{One weight conditions}
\end{center}
Combining \eqref{lattice Ap}, theorem \ref{Cp theorem}, \textit{remark \ref{remark1}}, \textit{remark \ref{remark2}}, \textit{remark \ref{remark fractional A infinity}} and theorem \ref{fractional A infinity and doubling}  we get, for $p<q$, the following lattice of inclusions for the conditions used in one weighted theory

\[ 
	\xy
	(-21,7)*+{A_1(\o)\subsetneq A_p(\o) \subsetneq A_q(\o) \subsetneq A_\infty(\o)\ \ \ \ \ };
	(3.2,10.5)*+{\rotatebox{45}{$\,\, \subsetneq\,\,$}};
	(26.5,13)*+{\A_\infty^\alpha(\o)\cap \mathcal{D}(\o)\subsetneq\left\{
	\begin{array}{l}
	\!\!\!\mathcal{D}(\o)
		  \\
	\!\!\!\A_\infty^\alpha(\o)	  
	\end{array}
\right.};
	(3.2,3)*+{\rotatebox{-45}{$\,\, \subsetneq\,\,$}};
	(23,1)*+{\ \ \ \ \ C_p(\o)\cap \mathcal{D}(\o)\subsetneq\left\{
	\begin{array}{l}
	\!\!\!\mathcal{D}(\o)
		  \\
	\!\!\!C_p(\o)	  
	\end{array}
\right.};
	\endxy    
	\]
\begin{center} 
\textbf{Two weight conditions}
\end{center}
Combining \textit{remark \ref{remark Ap}}, theorem \ref{non doubling Ap examples}, theorem \ref{Ap doubling equivalence theorem}, \textit{remark \ref{pivotal implies classical Ap}}, theorem \ref{non doubling pivotal example}, \textit{remark \ref{Pivotal implies Energy}}, corollary \ref{energy corollary}, theorem \ref{fractional A infinity and doubling}, theorem \ref{Sawyer testing theorem}, corollary \ref{A infinity and pivotal corollary} and the example in \cite{LSU} we get the following lattice of inclusions for the conditions used in two weighted theory.\\
\\
\text{For general Radon measures:}
\begin{eqnarray*}
\text{Theorem \ref{non doubling Ap examples},\,\,  }&&\\
\text{\textit{remark \ref{remark Ap}}, \textit{remark \ref{pivotal implies classical Ap}}, \cite{LSU}: }&&\V(\o,\s)^p\subsetneq \A_p(\o,\s)\subsetneq \A^{t_1}_p(\o,\s)\cup \A^{t_1}_p(\s,\o)\\
&& A_p^{t_1}(\o,\s)\cap A_p^{t_1}(\s,\o)=A_p^{t_2}(\o,\s)\\
&& A_p^{t_1}(\o,\s)\subsetneq A_p^{t_2}(\o,\s)\\
\text{\textit{Remark \ref{Pivotal implies Energy}},\,\, }\\
\text{theorem \ref{non doubling pivotal example}, corollary \ref{energy corollary}: }&&\E(\o,\s)^p\subsetneq\V(\o,\s)^p \centernot\implies \A_p^{t_1}(\o,\s)\subsetneq \A_p^{t_2}(\o,\s)
\end{eqnarray*}
\text{For doubling measures:}
\begin{eqnarray*}
\text{Theorem \ref{Ap doubling equivalence theorem}: }&&\A_p(\o,\s)=  A_p^{t_1}(\o,\s)= \A_p^{t_2}(\o,\s)\\
\text{Theorem \ref{Sawyer testing theorem}, corollary \ref{A infinity and pivotal corollary}: }&& \A_p(\o,\s)\cap \A_\infty(\o) \subsetneq  S_d(\o,\s)\subseteq\V(\o,\s)^p\\
\text{Theorem \ref{Ap and small doubling corollary}: }&& \A_p(\o,\s)\cap \mathcal{D}(\s)\cap\mathcal{D}(\o)\subsetneq \V(\o,\s)^p\\
&&\text{(small doubling constant)}
\end{eqnarray*}

\section{What we know so far}
The following diagram shows the relationships between the different conditions that have appeared in the study of two weighted inequalities for the $\textbf{1}$-testing case over the years.

\hspace{-1.3 cm}\includegraphics[height=7 cm,width=15 cm]{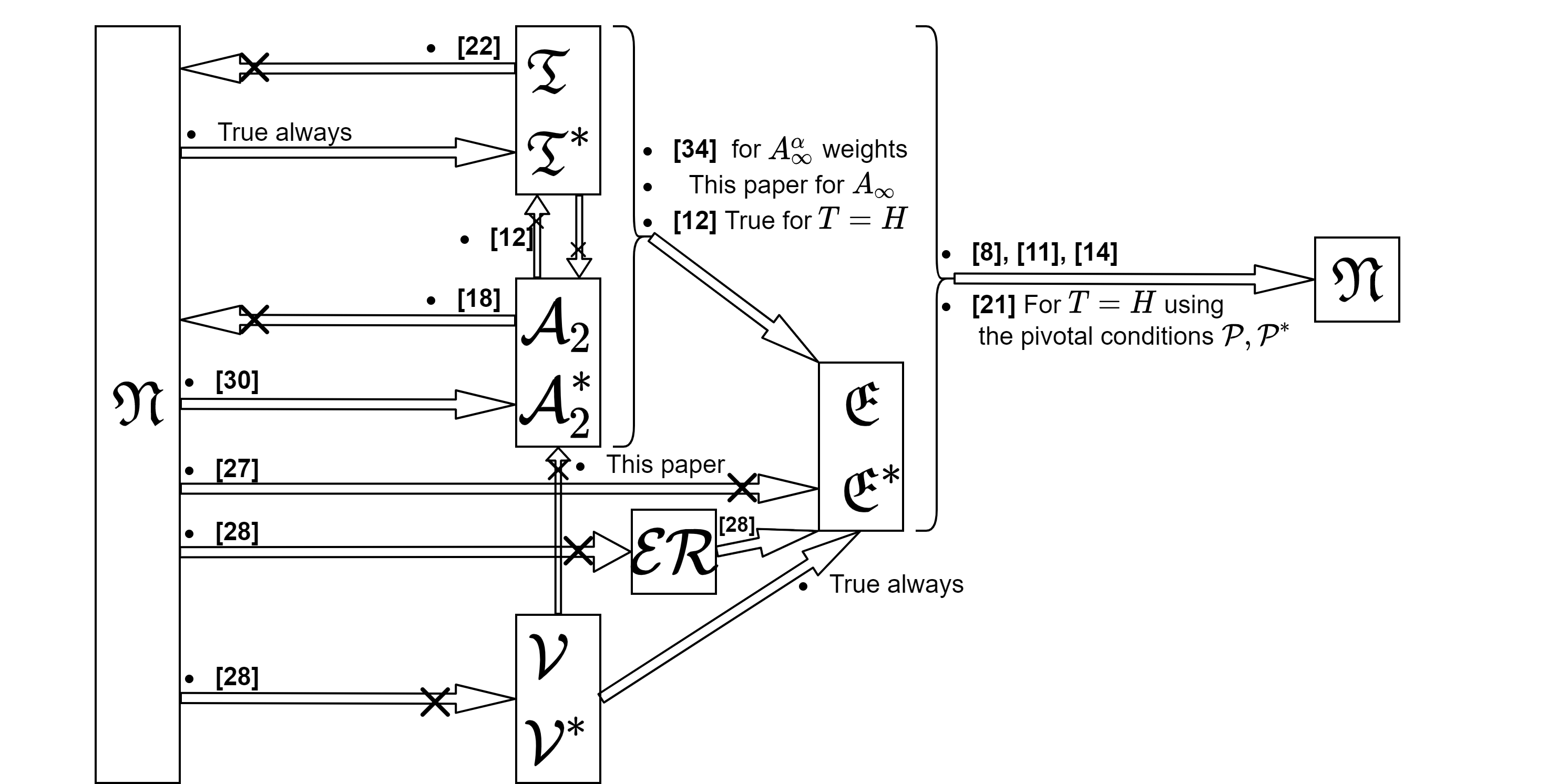}

\section{One weight conditions}
\subsection{The $A_1$ and $A_\infty$ conditions.}

We say the weight $w(x)$ is an $A_1$ weight if and only if 
\begin{equation}\label{one weight A1 condition}
Mw(x)\leq [w]_{A_1}w(x)
\end{equation}
and we call $[w]_{A_1}$ the $A_1$ constant of $w$. $A_1$ is a stronger condition than the $A_p$ condition for $p>1$.

If we take the union of all the $A_p$ weights for the different $p$ we get the larger class of $A_\infty$ weights, i.e. $A_\infty=\displaystyle \bigcup_{p>1}A_p$ (check \cite{DJ}, chapter 7). Another equivalent and commonly used characterization for $A_\infty$ weights is the following:
We say $w \in A_\infty$, if for all $I \subset \R^n$ and $E\subset I$, there exist uniform constants $C,\varepsilon>0$ such that
\begin{equation}
\label{A infinity}
\frac{w(E)}{w(I)}\leq C\left(\frac{|E|}{|I|}\right)^\varepsilon.
\end{equation}
\begin{rem}
We have the following linear lattice for $1<p<q<\infty$:
\begin{equation}\label{lattice Ap}
A_1\subsetneq A_p \subsetneq A_q \subsetneq A_\infty
\end{equation}
The power weights $w(x)=|x|^\alpha$ show that all the inclusions are proper.
\end{rem}
In particular we have the following known lemma.
\begin{lem}
Let $w(x)=|x|^\alpha$, $x \in \R^n$. Then\\
\begin{center}
$\displaystyle
[w]_{\A_p}\approx \begin{cases} (\alpha+n)^{-1}(-\alpha \frac{p\prime}{p}+n)^{-\frac{p}{p\prime}}, \quad -n<\alpha<n(p-1)
\\
\infty ,\quad \text{otherwise}
\end{cases}
$\end{center}\qed
\end{lem}
\subsection{Doubling and reverse doubling measures.} The $A_p$ weights for $1\leq p \leq \infty$ are all absolutely continuous to the Lebesgue measure and \textit{doubling}. 

We say a measure $\o$ on $\R^n$ is \textit{doubling}, and write $\o \in \mathcal{D}$, if it's not the zero measure and there 
is a constant $K>0$ such that for all cubes $I\subset \R^n$ we have 
\begin{equation}\label{doubling}
\o(2I)\leq K\o(I).  
\end{equation}

Not all doubling measures are $A_\infty$ as was first shown in \cite{FeMu} using an absolutely continuous measure $w$ that is also doubling but is not in $A_\infty$. Mutually singular doubling measures also exist, which of course are not in $A_\infty$, a nice construction can be found in \cite{GaKS}.

We say a measure $\s$ is \textit{reverse doubling} if there exists 
$\varepsilon>0$ depending only on the measure $\s$ such that for all cubes $I$:
\begin{equation}\label{reverse doubling}
\s(I)\leq (1+\varepsilon)\s(I).
\end{equation}
Doubling measures satisfy the reverse doubling property as the following lemma from \cite{Ruz} proves.
\begin{lem}\label{reverse doubling implies doubling}
Let $\s$ be a doubling measure with doubling constant $K_\s$. Then there exist a constant $\delta_\s>0$  depending only on the doubling constant of $\s$ such that for all cubes $I$ we have $\s(2I)\geq (1+\delta_\s)\s(I)$.\qed
\end{lem}
For the rest of this section, we are going to say that a measure $\o$ is doubling if 
\begin{equation}\label{doubling3}
\o(3I)\leq C\o(I).
\end{equation}
This definition is equivalent to \eqref{doubling}.

\subsection{A $C_p$ and doubling weight that is not in $A_\infty$}
In this subsection we give the proof for theorem \ref{Cp theorem}. The construction is a very involved variation of the construction in \cite{GaKS}.
\\
\\\textit{Proof of theorem \ref{Cp theorem}:} Let $I_0=[-\frac{1}{2},\frac{1}{2}]$ and $I_n=3I_{n-1}=3^nI_0$, the intervals 
centered at $0$ with length $3^n$. We call $\mathcal{G}$ the triadic grid created by the intervals $I_n$. Define the measure $w$ as follows: 
$w\left(x\right)=1$, $x \in I_0$ and $w\left(I_n\right)=\frac{1}{\delta_1^n}$, $\frac{1}{3}>\delta_1>0$ to 
be determined later. Call $I^l_n, I^m_n, I^r_n$ the left, middle and right third of 
$I_n$ respectively. Let $w(x)=\frac{3^{-n+1}(1-\delta_1)}{2\delta_1^n}$, $x \in I^r_n$.

Fix $k \in \N$ and $n_k \in \N$ to be determined later.
Let $I^{l,m}_{n_k}$ to have the same center as 
$I^l_{n_k}$ and $\vert I^{l,0}_{n_k}\vert=3^{m}$, $0\leq m \leq n_k-1$. Let $w(I^{l,m}_{n_k})=\delta_2^{n_k-m}w(I^l_{n_k})$, where $\frac{1}{3}>\delta_2>0$ to be determined later. For $m \geq 2$ let 
$w\left(x\right)=\frac{3(1-\delta_2)}{2|I^{l,m}_{n_k}|}w\left(I^{l,m}_{n_k}\right),$ for all $x\in I^{l,m}_{n_k}\backslash I^{l,m-1}_{n_k}$. This defines $w$ 
completely outside $3I^{l,0}_{n_k}$. Check Figure \ref{Figure 2.1.}.

$ $

$ $
\begin{center}
\begin{tikzpicture}
\label{Figure 2.1.}

\color{blue}
\draw (-2/27,0) -- (2/27,0);
\draw (-2/27-12/27,1/4) -- (-2/27,1/4);
\draw (2/27,1/4) -- (2/27+12/27,1/4);
\draw (-2/27-12/27-12/9,1/2) -- (-2/27-12/27,1/2);
\draw (2/27+12/27,1/2) -- (2/27+12/27+12/9,1/2);
\draw (2/27+12/27+12/9,1)--(2/27+12/27+12/9+12/3,1);
\color{red}
\draw (-2/27-12/27-12/9-12/3,1)--(-2/27-12/27-12/9-12/3+4/3,1);
\draw (-2/27-12/27-12/9-4/3,1)--(-2/27-12/27-12/9,1);
\draw (-2/27-12/27-12/9-12/3+4/3,0.5)--(-2/27-12/27-12/9-12/3+4/3+4/9,0.5);
\draw (-2/27-12/27-12/9-4/3-4/9,0.5)--(-2/27-12/27-12/9-4/3,0.5);
\draw (-2/27-12/27-12/9-12/3+4/3+4/9,0.25)--(-2/27-12/27-12/9-12/3+4/3+4/9+4/27,0.25);
\draw (-2/27-12/27-12/9-4/3-4/9-4/27,0.25)--(-2/27-12/27-12/9-4/3-4/9,0.25);
\draw (-2/27-12/27-12/9-12/3+4/3+4/9+4/27,0)--(-2/27-12/27-12/9-4/3-4/9-4/27,0);
\color{black}
\draw[decorate,decoration={brace,mirror}]  (-2/27-12/27-12/9-12/3,-0.1) -- (-2/27-12/27-12/9,-0.1)
node (m) at (-2/27-12/27-12/9-12/3+2,-0.4){\footnotesize $I^l_{k_n}$};
\draw[decorate,decoration={brace,mirror}]  (0,-0.1) -- (0,-0.1)
node (m) at (0,-0.4){\footnotesize $I_{0}$};
\end{tikzpicture}
Figure 4.3
\end{center}
$ $

Now let $I\subset I^{l,0}_{n_k}$ be any triadic interval such that
$|I|\geq 3^{-i_k}$,  and $i_k \in \N$ will be determined later. Let 
$$
w\left(I\right)=\left\{
	\begin{array}{ll}
		\delta_2 w\left(\pi I\right)  & \mbox{if } \partial I \cap \partial \pi 
		I=\emptyset\\
		\frac{1-\delta_2}{2}w\left(\pi I\right)  & \mbox{if } \partial I \cap 
		\partial \pi I\neq\emptyset
	\end{array}
\right.
$$
where $\pi I$ is the triadic parent of $I$ in the grid $\mathcal{G}$. Let $w\left(x\right)$ be constant for any triadic interval $I\subset I^{l,0}_{n_k}$ with $|I|\leq 3^{-i_k}$. 

We are left with defining $w$ on $3I^{l,0}_{n_k}\backslash I^{l,0}_{n_k}$. Call 
$J^l_{n_k}$ the left third of $3I^{l,0}_{n_k}$ and $J^r_{n_k}$ its right third. 
Let $J^{l,i_k}_{n_k}$ be the right most triadic $i_k$ child of $J^l_{n_k}$ and let
$w(x)=\left(\frac{1-\delta_2}{2}\right)^{i_k}w(J^{l}_{n_k})$, $x \in 
J^{l,i_k}_{n_k}$. Now for all triadic $I$ such that $J^{l,i_k}_{n_k} \subset I 
\subset J^{l}_{n_k}$, let $I^l, I^m, I^r$ denote the left, middle and right thirds
of $I$ and define $w(x)=\frac{3(1-\delta_2)}{2|I|}w(I)$, $x \in I^l$, 
$w(x)=\frac{3\delta_2}{|I|}w(I)$, $x\in I^m$ and $w(I^r)=\frac{1-\delta_2}{2}w(I)$. 
Similarly (but on the left end) we define $w$ on $J^r_{n_k}$. This construction on $J^l_{n_k}$ and $J^r_{n_k}$ is done so that $w$ is doubling.

Indeed, to see that $w$ is doubling, let $J_1,J_2$ be two triadic intervals of the 
same length that touch. If they have the same triadic parent then $w(J_1)/w(J_2)\lesssim 
\frac{1}{\min(\delta_1,\delta_2)}$. If not, we apply the first case to their common ancestor and get again $w(J_1)/w(J_2)\lesssim 
\frac{1}{\min(\delta_1,\delta_2)}$. For an arbitrary interval $I$, let $3^m\leq |I|\leq3^{m+1}$. Then $I \subset J_1\cup J_2$ triadic intervals with $|J_1|=|J_2|=3^{m+1}$. Then $w(3I)\lesssim \frac{1}{\min(\delta_1,\delta_2)}w(I)$.

Allowing $i_k \rightarrow \infty$ makes $w$ singular to Lebesgue. Check (\cite{GaKS}, Lemma 2.2).
Choose $i_k$ so that there exists an interval $J_{n_k}$ and  $E_{n_k}\subset J_{n_k} \subset 3I^{l,0}_{n_k}$ ,such that
\begin{equation}\label{worst case}
\frac{w\left(E_{n_k}\right)}{w\left(J_{n_k}\right)}\approx\frac{1}{2}, \quad \frac{|E_{n_k}|}{|J_{n_k}|}\approx \frac{1}{2^k} \quad \text{ and }\quad \frac{w\left(E\right)}{w\left(I\right)}\lesssim 2^k \frac{|E|}{|I|}
\end{equation}
for all intervals $I\subset 3I^{l,0}_{n_k}$ and $E\subset 
I$. This can be done by following in \cite{GaKS} definition 
2.1. and lemma 2.2. Note that because we stop at height 
$i_k$, \eqref{worst case} tells us that there is a ``worst 
interval" $J_{n_k}$.

\RC{\color{red} This is true by choosing $\delta_1,\delta_2<\frac{1}{3}$. It can be seen easily for intervals of the form $I=3^mI^{l,0}_{n_k}$ that 
$$
\frac{w\left(E\right)/w\left(I\right)}{|E|/|I|}\leq\left(3\min(\delta_1,\delta_2)\right)^m\frac{w\left(E\right)/w\left(I^{l,0}_{n_k}\right)}{|E|/|I^{l,0}_{n_k}|}\lesssim 2^k.
$$
Taking different cases on $I$ gives the answer for an arbitrary interval.}

By letting $k \rightarrow \infty$ it is clear that $A_\infty$  fails to hold for $w$. So we now need to prove that the $C_p$ condition holds.

By the end of the next calculation we will determine $\delta_1,\delta_2$. We want to prove $w$ is $C_p$ and for that we need to show that \eqref{Cp condition} holds with  $\int_\R \left(M\mathbf{1}_I\left(x\right)\right)^pw\left(x\right)dx<\infty$ for any interval $I$. Let first, $I=I^{l,0}_{n_k}$.

\begin{eqnarray}\label{Cp main}
&&
\int_\R \left(M\mathbf{1}_{I^{l,0}_{n_k}}\left(x\right)\right)^pw\left(x\right)dx=\int_{I^{l,0}_{n_k}}\left(M\mathbf{1}_{I^{l,0}_{n_k}}\left(x\right)\right)^pw\left(x\right)dx+\\
&+&
\int_{I^{l}_{n_k}\backslash I^{l,0}_{n_k}}\left(M\mathbf{1}_{I^{l,0}_{n_k}}\left(x\right)\right)^pw\left(x\right)dx+\int_{\R\backslash I^{l}_{n_k}}\left(M\mathbf{1}_{I^{l,0}_{n_k}}\left(x\right)\right)^pw\left(x\right)dx\notag\\
&\equiv&
A+B+C\notag
\end{eqnarray}
We have immediately $A=w\left(I^{l,0}_{n_k}\right)$, for $B$ we get
\begin{eqnarray*}
B
&=&\int_{I^{l}_{n_k}\backslash I^{l,0}_{n_k}}\left(\frac{|I^{l,0}_{n_k}|}{2|I^{l,0}_{n_k}|+2dist\left(x,I^{l,0}_{n_k}\right)}\right)^pw\left(x\right)dx\\ 
&\approx& 
2^{-p}\left(1-\delta_2\right)\sum_{m=1}^{n_k-1}\frac{3^{-mp}}{\delta_2^m}w\left(I^{l,0}_{n_k}\right)
=2^{-p}\left(1-\delta_2\right)w\left(I^{l,0}_{n_k}\right)\sum_{m=1}^{n_k-1}\left(\frac{3^{-p}}{\delta_2}\right)^m
\end{eqnarray*}
Now choose $\delta_2=\frac{3^{-p}}{2}$ so that the series above diverges (any $\delta_2\leq 3^{-p}$ works here). We also want $n_k$ so that
\begin{equation}\label{the gain}
2^{-p}\left(1-\delta_2\right)w\left(I^{l,0}_{n_k}\right)\sum_{m=1}^{n_k-1}\left(\frac{3^{-p}}{\delta_2}\right)^m\gtrsim 2^kw\left(I^{l,0}_{n_k}\right).
\end{equation}

We are only left with calculating term $C$. We have,
\begin{eqnarray}\label{tail converging}
C
&=&
\int_{\R\backslash I^{l}_{n_k}}\left(\frac{|I^{l,0}_{n_k}|}{2|I^{l,0}_{n_k}|+2dist\left(x,I^{l,0}_{n_k}\right)}\right)^pw\left(x\right)dx\\
&\approx&
2^{-p}3^{-n_kp}\left(1-\delta_1\right)\frac{1-\delta_2}{\delta_2^{n_k-1}}w\left(I^{l,0}_{n_k}\right)\sum_{m=1}^\infty \left(\frac{3^{-p}}{\delta_1}\right)^m\notag
\end{eqnarray}
Choose $\delta_1>3^{-p}$ so that the infinite series converges. Combining the estimates for $A,B$ and $C$ we get:
\begin{equation}\label{Cp first gain}
\int_\R \left(M\mathbf{1}_{I^{l,0}_{n_k}}\left(x\right)\right)^pw\left(x\right)dx<\infty 
\end{equation}
and
\begin{equation}\label{Cp first win}
\frac{w\left(E\right)}{\int_\R \left(M\mathbf{1}_{I^{l,0}_{n_k}}(x)\right)^pw\left(x\right)dx}\leq \frac{w\left(E\right)}{2^kw(I^{l,0}_{n_k})}\lesssim  \frac{2^k}{2^k}\frac{|E|}{|I^{l,0}_{n_k}|}=\frac{|E|}{|I^{l,0}_{n_k}|}.
\end{equation}
for $E\subset I^{l,0}_{n_k}$. 
We want to extend \eqref{Cp first gain} and \eqref{Cp first win} to all triadic 
intervals. Note that \eqref{Cp first gain} holds for any   interval $I$. To see that, choose $n$ big enough so that $I\subset I_n$. Then, following the calculations for estimating $C$ in \eqref{tail converging} we get that 
\begin{equation*}
\int_{\R\backslash I_n} \left(M\mathbf{1}_{I}(x)\right)^pw\left(x\right)dx<\infty 
\end{equation*}
which of course gives us
\begin{equation}\label{maximal finiteness}
\int_{\R}\left(M\mathbf{1}_{I}(x)\right)^pw\left(x\right)dx<\infty   
\end{equation}
To get \eqref{Cp first win} for any triadic $I\subset 
3I^{l,0}_{n_k}$, note that we can follow the same 
calculations that led to \eqref{the gain} and just choose $n_k$ big enough so that we get
the gain $2^kw(I)$. This is possible since the construction is finite and it stops at some height $i_k$. For that finite number of intervals, we choose $n_k$ big enough so that all the intervals get the gain $2^k$, i.e. $\int_\R \left(M\mathbf{1}_{I}(x)\right)^pw\left(x\right)dx\geq 2^kw(I)$.   So we have for any $E\subset I$, using \eqref{worst case},
\begin{equation}\label{Cp winning}
\frac{w\left(E\right)}{\int_\R \left(M\mathbf{1}_{I}(x)\right)^pw\left(x\right)dx}\leq \frac{w(E)}{2^kw(I)}\lesssim \frac{|E|}{|I|}.
\end{equation}
 
We will use the following calculation for triadic intervals $I\subset I^l_{n_k}$. Let $I=3I^{l,0}_{n_k}$, following \eqref{Cp main} and using $\delta_2=\frac{3^{-p}}{2}$,
$$
\int_\R(M1_{\mathbf{I}})^pdw\equiv A'+B'+C'.
$$ 
and $A'+B'\approx 3^p(A+B)$ hence
$$
\int_{I^l_{n_k}}\left(M\mathbf{1}_{I}\left(x\right)\right)^pw\left(x\right)dx\approx 3^p\int_{I^l_{n_k}} \left(M\mathbf{1}_{I^{l,0}_{n_k}}\left(x\right)\right)^pw\left(x\right)dx
$$
and
\begin{equation}\label{cp winning2}
\frac{w\left(E\right)}{\int_\R \left(M\mathbf{1}_{I}(x)\right)^pw\left(x\right)dx}\lesssim \frac{w(E)}{3^p2^kw(I^{l,0}_{n_k})}\lesssim 3^{1-p}\frac{|E|}{|I|}\leq \frac{|E|}{|I|}
\end{equation}
for any $E\subset 3I^{l,0}_{n_k}$, so we don't lose any of 
the ``gain" necessary for \eqref{Cp winning} to hold. We can
repeat for all triadic intervals $I$ such that 
$I^{l,0}_{n_k}\subset I\subset I^l_{n_k}$. Note that for $I=I^l_{n_k}$ $B'=0$. To extend \eqref{cp winning2} to triadic intervals $I\supset I^l_{n_k}$ notice that 
$$
\frac{w(E)}{w(\pi(I))}\lesssim\delta_1\frac{w(E)}{w(I)}\lesssim \delta_1\frac{|E|}{|I|}\leq 3\delta_1\frac{|E|}{|\pi(I)|}\leq \frac{|E|}{|\pi(I)|}\Longrightarrow \frac{w(E)}{w(\pi(I))}\lesssim \frac{|E|}{|\pi(I)|}
$$
for any $E\subset 3I^{l,0}_{n_k}$, where we used $\delta_1<\frac{1}{3}$.

To get \eqref{Cp winning} for an arbitrary triadic interval, let $I$ be a triadic interval not contained in any $I^{l,0}_{n_k}$ and $E$ any subset of $I$. We write 
$$
E=\left(\bigcup_{I^{l,0}_{n_k}\subset I}\left(E\cap 3I^{l,0}_{n_k}\right)\right)\bigcup \left(E\big\backslash\bigcup_{I^{l,0}_{n_k}\subset I} 3I^{l,0}_{n_k}\right)=E_1\cup E_2
$$
Using \eqref{cp winning2} we see that
\begin{eqnarray}\label{end proof1}
w(E_1)=\sum_{I^{l,0}_{n_k}\subset I}w(E\cap 3I^{l,0}_{n_k})
&\lesssim&
\sum_{I^{l,0}_{n_k}\subset I}\frac{|E\cap 3I^{l,0}_{n_k}|}{|I|}\int_\R(M1_{\mathbf{I}})^pdw\\
&=&\frac{|E_1|}{|I|}\int_\R(M1_{\mathbf{I}})^pdw \notag
\end{eqnarray}
To deal with $E_2$, note that for $\displaystyle x \in I\big\backslash\bigcup_{I^{l,0}_{n_k}\subset I} 3I^{l,0}_{n_k}$, $w(x)\lesssim \frac{3(1-\delta_2)}{2|I|}w(I)$ so we get
\begin{equation}\label{end proof2}
 \frac{w(E_2)}{w(I)}\approx \frac{|E_2|}{|I|}
\end{equation}
combining \eqref{end proof1}, \eqref{end proof2} we get \eqref{Cp winning} for a 
triadic interval $I$. 

We are left with extending \eqref{Cp winning} to an arbitrary interval $I$. Let $3^m\leq |I|\leq 3^{m+1}$ and $E\subset I$. Then $I \subset J_1\cup J_2$, $J_1,J_2$ triadic intervals such that $|J_1|=|J_2|=3^{m+1}$.  Since
$$
M\mathbf{1}_{I}(x)\approx M\mathbf{1}_{J_1}(x)\approx M\mathbf{1}_{J_2}(x), \text{for all }x \in \R.
$$
we get
\begin{eqnarray*}
\frac{w\left(E\right)}{\int_\R \left(M\mathbf{1}_{I}(x)\right)^pw\left(x\right)dx}
&\approx& 
\frac{w\left(E\cap J_1\right)}{\int_\R \left(M\mathbf{1}_{J_1}(x)\right)^pw\left(x\right)dx}+
\frac{w\left(E\cap J_2\right)}{\int_\R \left(M\mathbf{1}_{J_2}(x)\right)^pw\left(x\right)dx}\\
&\lesssim&
\frac{|E\cap J_1|}{|J_1|}+\frac{|E\cap J_2|}{|J_2|}\approx \frac{|E|}{|I|}
\end{eqnarray*}
This shows that $w$ satisfies \eqref{Cp condition} and hence $w$ is a $C_p$ 
weight and the proof is complete. \qed

\subsection{Doubling $C_p$ weights are in $\A_\infty$ for small doubling constants} 

Note that the construction in the proof of theorem \ref{Cp theorem} we depended heavily on the big doubling constant of the weight $w$. Here we show that this is the only case by proving theorem \ref{Cp theorem small doubling}. 
\\
\\\textit{Proof of theorem \ref{Cp theorem small doubling}:} It will be enough to show that 
$$
\int_{\R^n}|M\mathbf{1}_I|^pw(x)dx\approx w(I)
$$
the result then follows immediately from \eqref{Cp condition}. Let $I_n=3^mI$ the 
cubes with same center as $I$ and side length $\ell(I_n)=3^m\ell(I)$. We write
\begin{eqnarray*}
\int_{\R^n}|M\mathbf{1}_I|^pwdx
\!\!\!&=&\!\!\!
\sum_{m=0}^\infty \int_{I_m\backslash I_{m-1}}\!\!\!\!\!\!\!|M\mathbf{1}_I|^pwdx\approx\sum_{m=0}^\infty \int_{I_m\backslash I_{m-1}}\frac{|I|^pw(x)dx}{(|I|^\frac{1}{n}+\dist(x,I))^{np}}\\
&\approx&\!\!\!
\sum_{m=0}^\infty |I|^p\frac{w(I_m)}{|I_m|^p}\lesssim \sum_{m=0}^\infty \frac{(C_w)^mw(I)}{(3^{np})^m}\lesssim w(I)
\end{eqnarray*}
since $C_\s<3^{np}$ by hypothesis and the series converges.\qed

\begin{rem}\label{remark1}
Not all doubling weights are $C_p$ weights. For an example just choose $\delta_1<3^{-p}$ in the construction of theorem \ref{Cp theorem}.
\end{rem}

\begin{rem}\label{remark2}
There exist non-doubling $C_p$ weights. For an example choose $\delta_{2,k}=\frac{1}{5k}$ in each $I_{n_k}$  in the construction of theorem \ref{Cp theorem}. A much simpler example is given by getting the Lebesgue measure in $\R^n$ and setting the measure of the unit ball equal to 0, i.e. define $w(E)=m(E\backslash B(0,1))$ where $B(0,1)$ is the unit ball in $\R^n$.
\end{rem}
\subsection{The $\A_\infty^\alpha$ condition.}

To complete the picture for the one weight conditions we are introducing the fractional $\A_\infty^\alpha$ condition. We are following very closely \cite{Saw2} where $\A_\infty^\alpha$ was introduced.

First we define the $\alpha-$relative capacity of a measure $\mathbf{Cap}_\alpha(E;I)$ of a compact subset $E$ of a cube $I$ by

$$
\mathbf{Cap}_\alpha(E;I)\!=\!\inf \Big\{\!\!\int h(x)dx:h \geq 0, Supp h\subset 2I \text{ and }I_\alpha h\geq (diam2I)^{\alpha -n} \text{ on }E\Big\}
$$
Check \cite{AH} for more properties on capacity.

We say that a locally finite positive Borel measure $\o$ is said to be an $\A_\infty^\alpha$ measure if
\begin{equation}\label{A alpha infinity}
\dfrac{\o(E)}{\o(2I)}\leq \eta(\mathbf{Cap}_\alpha(E,I))
\end{equation}
when $\o(2I)>0$, for all compact subsets E of a cube I, for some function $\eta : [0,1] \to  [0,1]$ with $\displaystyle\lim_{t \to 0}\eta(t)=0$. 

Note that omitting the factor 2 in $\o(2I)$ above makes the condition more restrictive in general, but remains equivalent for doubling measures. It is shown in \cite{Saw2} that $\o \in \A_\infty^\alpha$ implies the Wheeden-Muckenhoupt inequality 
\begin{equation} 
\int\left|I_\alpha f\right|^pd\o\leq \int \left|M_\alpha f\right|^pd\o
\end{equation}
for all $f$ positive Borel measures.
\begin{rem}\label{remark fractional A infinity}
$\A_\infty^\alpha$ measures are not necessarily doubling. 
Take for example the Lebesgue measure in $\R^n$ and set the 
measure of the unit ball equal to 0, i.e. define 
$\o(E)=m(E\backslash B(0,1))$ where $B(0,1)$ is the unit ball
in $\R^n$. This measure is clearly non-doubling and hence 
not in $A_\infty$ but it is an $\A_\infty^\alpha$ measure.
\end{rem}
There exist also doubling fractional $A_\infty$ measures that are not in $A_\infty$. The example we use for that is exactly the one used in \cite{GaKS} but here we have to calculate the relative capacities of the sets used.

\begin{thm}\label{fractional A infinity and doubling}
($\A^\alpha_\infty \cap \mathcal{D}\nRightarrow A_\infty $) There exist a measure $\mu$ singular to the Lebesgue measure that is doubling and satisfies the $\A^\alpha_\infty$ condition with $\eta(t)=t$ but $\mu$ is not an $A_\infty$ weight.
\end{thm}
\begin{proof}

Let $\mu([0,1])=1$, $0<\delta<3^{-1}$ to be determined later, and for any triadic $I \subset [0,1]$ let

$$
\mu\left(I\right)=\left\{
	\begin{array}{ll}
		\delta \mu\left(\pi I\right)  & \mbox{if } \partial I \cap \partial \pi 
		I=\emptyset\\
		\frac{1-\delta}{2}\mu\left(\pi I\right)  & \mbox{if } \partial I \cap 
		\partial \pi I\neq\emptyset
	\end{array}
\right.
$$
It was shown in \cite{GaKS} that $\mu$ is a doubling measure. It is also shown that it is singular to the Lebesgue measure hence it does not satisfy the $A_\infty$ condition. 

To show that it satisfies the $\A^\alpha_\infty$ condition, let $I\subset [0,1]$ be a triadic interval and $E\subset I$ be compact. 

We claim that $||I_\alpha\mu_I||_{L^\infty(I)}=C_{\alpha,\delta}\mu(I)|I|^{\alpha-1}$, where $\mu_I$ is the restriction of $\mu$ on the set $I$ and the constant $C_{\alpha,\delta}$ is independent of $I$. For any $x \in I$ we have
$$
I_\alpha \mu_I(x)
=\int_I |x-y|^{\alpha-1}d\mu(y)\lesssim 
\mu(I)|I|^{\alpha-1}\sum_{k=0}^\infty 3^{k(1-\alpha)}\left(\frac{1-\delta}{2}\right)^k=C_{\alpha,\delta}\mu(I)|I|^{\alpha-1}
$$
as long as $3^{1-\alpha}\frac{1-\delta}{2}<1 \Rightarrow \alpha>1-\frac{\ln(\frac{2}{1-\delta})}{\ln 3}$.

Now for any $f \geq 0$, $Suppf \subset 2I$ and $I_\alpha f\geq |2I|^{\alpha-1}$ on $E$, using Fubini's theorem we have 
\begin{eqnarray*}
\mu(E)
\!\!\!\!&=&\!\!\!\!
\int_I\mathbf{1}_Ed\mu\leq \int_I|2I|^{1-\alpha}I_\alpha f(x)\mathbf{1}_E(x)d\mu(x)=\int |2I|^{1-\alpha}I_\alpha\mu_E(x)f(x)dx\\
&\leq&\!\!\!\!
||f||_1||I_\alpha\mu_E||_\infty|2I|^{1-\alpha}
\leq ||f||_1||I_\alpha\mu_I||_\infty|2I|^{1-\alpha}\lesssim ||f||_1\mu(I)
\end{eqnarray*}
So $Cap_\alpha(E,I)\gtrsim \frac{\mu(E)}{\mu(I)}$ hence $\A^\alpha_\infty$ holds with $\eta(t)=t$ and the proof is complete.
\end{proof}

\section{Two weight conditions}
We start this section with the proofs of theorems \ref{non doubling Ap examples} and \ref{Ap doubling equivalence theorem}.

\subsection{Non doubling $\A_p$ examples}
\begin{rem}\label{remark Ap}
Note first that we have the following simple implications  $\A_p^{t_2} \Rightarrow \A_p^{t_1} \Rightarrow \A_p$. Indeed it is easy to see:
\begin{eqnarray*}
P(I,\s)&=&\int_I\displaystyle\frac{\displaystyle|I|}{\displaystyle\left(|I|+\dist(x,I)\right)^2}\s(dx)+\int_{\R/I}\frac{|I|}{\displaystyle\left(|I|+\dist(x,I)\right)^2}\s(dx)\\
&=&
\frac{\s (I)}{|I|}+\int_{\R/I}\frac{|I|}{\displaystyle\left(|I|+\dist(x,I)\right)^2}\s(dx)\geq \frac{\s(I)}{|I|}
\end{eqnarray*}
and so immediately from the definitions \eqref{classical Ap}, \eqref{one tailed} and \eqref{two tailed} we get 
$$
\A_p(\o,\s)\subseteq \A_p^{t_1}(\o,\s) \subseteq \A_p^{t_2}(\o,\s) .
$$
\end{rem}
\begin{rem}We work with $p=2$ for simplicity. The examples we use work with trivial modifications for any $p>1$. 
\end{rem}
$ $\\
\textit{Proof of theorem \ref{non doubling Ap examples}:}

\textit{(1)} We want to construct two measures $\o,\s$ such that the two weight classical $\A_2$ condition holds but both one tailed $\A_2$ conditions fail. First, we construct measures $u_k$ and $v^n_k$ that satisfy
$$
\frac{u_k(I)v^n_k(I)}{|I|^2}\leq M,
$$
$$
\frac{u_k(I)}{|I|}P(I,v^n_k)\gtrsim n
$$
where the constant $M$ does not depend on $k,n$. Then we will combine the measures $u_k$ and $v^n_k$ to create $\o,\s$ such that the two weight classical $\A_2$ condition holds and both one tailed $\A_2$ conditions fail for the pair $\o,\s$. 
Let
$$
u_k(E)=m(E \cap [k,k+1])
, \quad  
v^n_k(E)=\sum_{i=0}^n \displaystyle 2^{i}m\left(E \cap \left[k+2^i,k+2^{i+1}\right]\right)
$$
where $m$ is the classic Lebesgue measure on $\R$. Let $I=(a,b)$, 
$a < k+1$ and $k+2^{i-1} \leq b <k+2^i$, for some $i\geq 0$ (of course if the interval does not 
intersect $[k,k+1]$ then $u_k(I)=0$). Then 

\begin{equation}\label{classical holds}
\frac{u_k(I)v^n_k(I)}{|I|^2} \leq \frac{4^{i+1}-1}{(4-1)(2^{i-1}-2)^2}=M
\end{equation}
which is bounded for $i>2$ (the cases $i=0,1,2$ can be seen directly). Now let $I=[k,k+1]$. We have then:

$$\frac{u_k(I)}{|I|}P(I,v^n_k)=\int_\R \frac{v^n_k (dx)}{\displaystyle\left(1+\dist(x,[k,k+1])\right)^2}=\sum_{i=0}^n \int_{I_i} \frac{v^n_k (dx)}{\displaystyle\left(1+\dist(x,[k,k+1])\right)^2}$$
where $I_i=[k+2^i,k+2^{i+1}]$. We get:

\begin{eqnarray}\label{tail lose}\hspace{0.5 cm}\displaystyle\sum_{i=0}^k \int_{I_i} \frac{v^n_k (dx)}{\displaystyle\left(1+\dist(x,[1,2])\right)^2} &\geq& 1+\displaystyle\sum_{i=1}^n \int_{I_i} \frac{v^n_k (dx)}{\displaystyle\left(1+2^i-2)\right)^2}\\
&=&1+\displaystyle\sum_{i=1}^n\frac{v^n_k(I_i)}{\left(2^i-1\right)^2}=1+\sum_{i=1}^n\frac{2^{2i}}{\left(2^i-1\right)^2}\approx n.\notag
\end{eqnarray}
Now we define $\o,\s$ as follows:
$$
\o(E)=\sum_{k=1}^\infty u_{100^k}(E)+\sum_{k=1}^\infty v^k_{-100^k}(E)
$$
$$
\s(E)=\sum_{k=1}^\infty u_{-100^k}(E)+\sum_{k=1}^\infty v^k_{100^k}(E)
$$

It is easy to see that with $I=I_k=[100^k,100^k+1]$ both one tailed $\A_2$ conditions fail using \eqref{tail lose}.

To see that the classical $\A_2$ condition holds, let $I=(a,b)$ be any interval. It is simple to check that if $I$ is big enough such that  $|a|\approx 100^k, |b|\approx 100^n$ for $k\neq n$ then 
$$
\frac{\o(I)\s(I)}{|I|^2}\leq 1.
$$
While if $|a|\approx |b|\approx 100^k$ for some $k$ then using \eqref{classical holds} we get 
$$
\frac{\o(I)\s(I)}{|I|^2} \leq 2M
$$
hence the classical two weight $\A_2$ condition holds but both one tailed $\A_2$ conditions fail.
\\

\textit{(2)} Now we turn to proving $\A^{t_1}_2  \nRightarrow \A^{t_2}_2$. Let the new measures be:
\begin{eqnarray*}
\o(E)&=&\displaystyle\sum_{n=1}^\infty 2^nm\left(E \cap \left[2^n,2^{n+1}\right]\right)\\
\s(E)&=&m(E \cap [0,1])
\end{eqnarray*}
From the construction above we can see that with $I=[0,1]$ we get:
\begin{eqnarray*}&&P(I,\o)P(I,\s)=\int_\R \frac{\o (dx)}{\displaystyle\left(1+\dist(x,[0,1])\right)^2}\int_\R \frac{\s (dx)}{\left(1+\dist(x,[0,1])\right)^2}\\
&=&
\int_\R \frac{\o (dx)}{\left(1+\dist(x,[0,1])\right)^2}\int_0^1 \frac{dx}{\left(1+\dist(x,[0,1])\right)^2}
=
\int_\R \frac{\o (dx)}{\left(1+\dist(x,[0,1])\right)^2}
\end{eqnarray*}
Now from the definition of $\o$ the last expression is equal to:
$$\displaystyle\sum_{n=1}^\infty \int_{2^n}^{2^{n+1}}\frac{2^n}{\left(1+\dist(x,[0,1])\right)^2}dx\geq \sum_{n=1}^\infty\frac{2^{2n}}{2^{2n}}=\infty$$
To prove that $\A_2^{t_1}$ hold let $I$ be an interval such that $2^n \leq |I|<2^{n+1}$ and $2^k-1 \leq \dist(I,[0,1])<2^{k+1}-1$ with $k \geq 0$. We have two cases:

(i) $n\geq k$.
$$\displaystyle\frac{\o(I)}{|I|}P(I,\s)\leq \displaystyle\frac{\displaystyle\sum_{l=1}^{n+1}2^lm\left(I\cap \left[2^l,2^{l+1}\right]\right)}{|I|^2}\leq \displaystyle\frac{\displaystyle\sum_{l=1}^{n+1}2^{2l}}{2^{2n}}=\frac{2^{2(n+2)}-1}{(4-1)2^{2n}}<M<\infty$$
where the first inequality uses the fact that the interval cannot intersect any point in $[2^{n+2}, \infty)$ otherwise $n \geq k$ would not be satisfied.

(ii)$n<k$. If $k=0$ then $I \cap [2,\infty)=\emptyset$ and there is nothing to prove. So assume $k>0$.
$$\frac{\o(I)}{|I|}P(I,\s)\leq \displaystyle\frac{\displaystyle\sum_{l=k}^{k+1}2^lm\left(I\cap \left[2^l,2^{l+1}\right]\right)}{2^{2k}}\leq \frac{2^{2k}+2^{2(k+1)}}{2^{2k}}=5<\infty$$
where the first inequality now holds because $I$ cannot contain neither any point in $ (0,2^k)$ for otherwise $\dist(I,(0,1))<2^k-1$ nor  any point in $[2^{k+1},\infty)$ because $n<k$ would not be satisfied and the proof is complete.
\\
\textit{(3)} Last, for the equivalence of the two tailed $\A_2$ condition to both one tailed $\A_2$ conditions let $I\in \R^n$ be a cube. We have:

\begin{multicols}{2}

$$
P(I,\s)\approx \frac{\s(I)}{|I|}+\sum_{k=1}^\infty \sum_{m=1}^{3^n-1} \frac{\s(I^k_m)}{3^{kn}|I^k_m|}
$$

$$
P(I,\o)\approx \frac{\o(I)}{|I|}+\sum_{k=1}^\infty \sum_{m=1}^{3^n-1} \frac{\o(I^k_m)}{3^{kn}|I^k_m|}
$$

\columnbreak
\begin{center}
\resizebox{3 cm}{3 cm}{%
 \begin{tikzpicture}
 
\label{Figure 2.2.}
\node at (0,0) {I};
\node at (1/2,1/2) {$I^1_m$};
\node at (3/2,3/2) {$I^2_m$};
\color{blue}
\draw (-1/4,-1/4) -- (-1/4,1/4);
\draw (-1/4,1/4)-- (1/4,1/4);
\draw (1/4,1/4)-- (1/4,-1/4);
\draw (1/4,-1/4)-- (-1/4,-1/4);
\draw (-3/4,-3/4) -- (-3/4,3/4);
\draw (-3/4,3/4)-- (3/4,3/4);
\draw (3/4,3/4)-- (3/4,-3/4);
\draw (3/4,-3/4)-- (-3/4,-3/4);
\draw (-1/4,-1/4) -- (-3/4,-1/4);
\draw (-1/4,1/4)-- (-3/4,1/4);
\draw (1/4,1/4)-- (3/4,1/4);
\draw (1/4,-1/4)-- (3/4,-1/4);
\draw (-1/4,-1/4) -- (-1/4,-3/4);
\draw (-1/4,1/4)-- (-1/4,3/4);
\draw (1/4,1/4)-- (1/4,3/4);
\draw (1/4,-1/4)-- (1/4,-3/4);
\draw (-9/4,-9/4) -- (-9/4,9/4);
\draw (-9/4,9/4)-- (9/4,9/4);
\draw (9/4,9/4)-- (9/4,-9/4);
\draw (9/4,-9/4)-- (-9/4,-9/4);
\draw (-3/4,-3/4) -- (-9/4,-3/4);
\draw (-3/4,3/4)-- (-9/4,3/4);
\draw (3/4,3/4)-- (3/4,9/4);
\draw (3/4,-3/4)-- (9/4,-3/4);
\draw (-3/4,-3/4) -- (-3/4,-9/4);
\draw (-3/4,3/4)-- (-3/4,9/4);
\draw (3/4,3/4)-- (9/4,3/4);
\draw (3/4,-3/4)-- (3/4,-9/4);
\end{tikzpicture}
}

    Figure 5.1
\end{center}

\end{multicols}

where $|I^k_m|^\frac{1}{n}=3^k|I|^\frac{1}{n}$ and $d(I^k_m,I)\approx 3^k $, and all the implied constants depend only on the dimension, check  Figure \ref{Figure 2.2.}. There exist $k_1,k_2\geq 0$ such that 

$$
P(I,\s)\approx 2\left(\frac{\s(I)}{|I|}+\sum_{k=1}^{k_1}\sum_{m=1}^{3^n-1} \frac{\s(I^k_m)}{3^{kn}|I^k_m|}\right)\approx 2 \sum_{k=k_1}^\infty \sum_{m=1}^{3^n-1} \frac{\s(I^k_m)}{3^{kn}|I^k_m|}
$$
$$
P(I,\o)\approx 2\left(\frac{\o(I)}{|I|}+\sum_{k=1}^{k_2}\sum_{m=1}^{3^n-1} \frac{\o(I^k_m)}{3^{kn}|I^k_m|}\right)\approx 2 \sum_{k=k_2}^\infty \sum_{m=1}^{3^n-1} \frac{\o(I^k_m)}{3^{kn}|I^k_m|}
$$

We can assume without loss of generality that $k_1\leq k_2$. Let $J=\displaystyle I\cup\left(\bigcup_{k=1}^{k_1}\bigcup_{m=1}^{3^n-1}I^k_m\right)$, hence $|J|^\frac{1}{n}\approx 3^{k_1}|I|^\frac{1}{n}$ where again the implied constant depends only on dimension. We calculate

\begin{eqnarray*}
\frac{\s(J)}{|J|}P(J,\o)
&\approx&
\frac{1}{|J|}\left(\s(I)+\sum_{k=1}^{k_1}\sum_{m=1}^{3^n-1}\s(I^k_m)\right)\left(\frac{\o(J)}{|J|}+\sum_{k=1}^\infty\sum_{m=1}^{3^n-1}\frac{\o(J^k_m)}{3^{kn}|J^k_m|}\right)\\
&\approx& 
\frac{1}{3^{k_1n}|I|}\left(\s(I)+\sum_{k=1}^{k_1}\sum_{m=1}^{3^n-1}\s(I^k_m)\right)\left(\frac{\o(J)}{|J|}+\sum_{k=1}^\infty\sum_{m=1}^{3^n-1}\frac{\o(I^{k+k_1}_m)}{3^{kn}|I^{k+k_1}_m|}\right)\\
&\approx& 
\frac{1}{3^{k_1n}|I|}\left(\s(I)+\sum_{k=1}^{k_1}\sum_{m=1}^{3^n-1}\s(I^k_m)\right)\left(\frac{\o(J)}{|J|}+\sum_{k=k_1}^\infty\sum_{m=1}^{3^n-1}\frac{3^{k_1n}\o(I^{k}_m)}{3^{kn}|I^{k}_m|}\right)\\
&\gtrsim& 
\frac{1}{|I|}\left(\s(I)+\sum_{k=1}^{k_1}\sum_{m=1}^{3^n-1}\s(I^k_m)\right)\left(\frac{\o(J)}{|J|}+\sum_{k=k_2}^\infty\sum_{m=1}^{3^n-1}\frac{\o(I^{k}_m)}{3^{kn}|I^{k}_m|}\right)\\
&\gtrsim& 
\left(\frac{\s(I)}{|I|}+\sum_{k=1}^{k_1}\sum_{m=1}^{3^n-1}\frac{\s(I^k_m)}{3^{kn}|I^k_m|}\right)\left(\frac{\o(J)}{|J|}+\sum_{k=k_2}^\infty\sum_{m=1}^{3^n-1}\frac{\o(I^{k}_m)}{3^{kn}|I^{k}_m|}\right)\\
&\approx&
P(I,\s)P(I,\o)
\end{eqnarray*}
hence showing that the one tailed $\A_p$ conditions bound the two tailed $\A_p$ condition and the proof is complete.

\qed
\begin{rem}
From the above construction we see that the same measures could work to prove the same implications for $\A_p^{\text{offset}}$ \eqref{offset A2}, and it's two tailed analogue since it's exactly the nature of the tail that we take advantage of in the construction.
\end{rem}

\subsection{Two weight $\A_p$ equivalence for doubling measures}
\begin{rem}We are going to use $p=2$ in the proof for simplicity. The general case follows immediately since $\frac{1}{p},\frac{1}{p'}<1$ and hence $$P(I,\o)^\frac{1}{p}\approx \left(\sum_{k=1}^\infty\sum_{j=1}^{3^n-1}\frac{\o(I^k_i)}{3^{2kn}|I|}\right)^\frac{1}{p}\leq \sum_{k=1}^\infty\sum_{j=1}^{3^n-1}\left(\frac{\o(I^k_i)}{3^{2kn}|I|}\right)^\frac{1}{p}$$
and from here the proof follows the same way as for $p=2$.
\end{rem}

\textit{Proof of theorem \ref{Ap doubling equivalence theorem}:} Let $\o,\s$ be reverse doubling measures with reverse doubling constants $1+\delta_\o$ and $1+\delta_\s$ respectively. It is enough to prove that we can bound the two tailed $\A^{t_2}_p(\o,\s)$ from the classical $\A_p(\o,\s)$. Let $I$ be a cube. We then have,
\begin{eqnarray*}
P(I,\o)P(I,\s)\!\!\!
&\lesssim&\!\!\!\!
\frac{\o(I)\s(I)}{|I|^2}
+
\frac{\o(I)}{|I|}\sum_{m=1}^\infty\sum_{i=1}^{3^n-1}\frac{\s(I^m_i)}{3^{2mn}|I|}+\frac{\s(I)}{|I|}\sum_{k=1}^\infty\sum_{j=1}^{3^n-1}\frac{\o(I^k_i)}{3^{2kn}|I|}\\
\!\!\!&+&\!\!\!\!
\sum_{m=1}^\infty\sum_{i=1}^{3^n-1}\frac{\s(I^m_i)}{3^{2mn}|I|}\sum_{k=1}^\infty\sum_{j=1}^{3^n-1}\frac{\o(I^k_j)}{3^{2kn}|I|}\equiv A+B+C+D
\end{eqnarray*}
where $|I^m_j|=3^{mn}|I|$, $\dist(I^m_j,I)\approx 3^m|I|^\frac{1}{n}$, $\displaystyle\bigcup_{m \in \N}\bigcup_{j=1}^{3^n-1}I^m_j=\R^n\backslash I$ and the implied constant depends only on dimension. $A$ is bounded immediately by $\A_2(\o,\s)$. For $B$ we have:

$$
B=\sum_{m=1}^\infty\sum_{i=1}^{3^n-1}\frac{\o(I)\s(I^m_i)}{3^{2mn}|I|^2}\lesssim \sum_{m=1}^\infty(1+\delta_\o)^{-m}\sum_{i=1}^{3^n-1}\frac{\o(I^m)\s(I^m)}{|I^m|^2}\lesssim \A_2(\o,\s)<\infty 
$$
where $\displaystyle I^m=I\cup 
\left(\bigcup_{\ell=1}^{m}\bigcup_{j=1}^{3^n-1}I^\ell_j\right)$ and the 
implied constant again depends only on dimension and the reverse doubling 
constant of $\o$. The bound for $C$ is similar to $B$. 

For $D$ we have:
\begin{eqnarray*}
D
&=&
\sum_{m=1}^\infty\sum_{k=1}^m\sum_{i=1}^{3^n-1}\sum_{j=1}^{3^n-1}\frac{\s(I^m_i)\o(I^k_j)}{3^{2mn}|I|3^{2kn}|I|}+\sum_{k=1}^\infty\sum_{m=1}^{k-1}\sum_{j=1}^{3^n-1}\sum_{i=1}^{3^n-1}\frac{\s(I^m_i)\o(I^k_j)}{3^{2mn}|I|3^{2kn}|I|}\\
&\equiv& \mathbf{I}+\mathbf{II}
\end{eqnarray*}
We will get the bound for $\mathbf{I}$, the calculations for $\mathbf{II}$ are identical.
\begin{eqnarray*}
\mathbf{I}&\lesssim&
\sum_{m=1}^\infty\sum_{k=1}^m\sum_{i=1}^{3^n-1}\sum_{j=1}^{3^n-1}(1+\delta_\o)^{k-m}\frac{\s(I^m)\o(I^m)}{3^{2kn}|I^m|^2}\lesssim\\
&\lesssim&
\A_2(\o,\s) \sum_{m=1}^\infty (1+\delta_\o)^{-m}\sum_{k=1}^m\frac{(1+\delta_\o)^k}{3^{2kn}}
\leq C_{n,\s}\A_2(\o,\s) <\infty
\end{eqnarray*}
Combining all the above bounds and getting supremum over the cubes $I$ we get
$$
\A_2^{t_2}(\o,\s)\leq C_{n,\o,\s}\A_2(\o,\s)
$$
which completes the proof of the theorem.\qed

\begin{rem}
The same proof works for the fractional $\A_p(\o,\s)$ conditions as defined in \cite{SSU4}.
\end{rem}

\subsection{The $T_1$ theorem for $A_\infty$ weights.}
The goal of this subsection is to prove theorem \ref{T1 theorem}. For that we are going to use the Sawyer testing condition.
\subsubsection{The Sawyer testing condition.}
Sawyer in \cite{Saw4} proved that the Maximal operator is bounded on $L^p(u)\to L^q(w)$ if and only if the Sawyer testing condition is satisfied, i.e. if and only if
\begin{equation}\label{Sawyer testing}
   S^{p,q}(w,u^{1-p'})=\sup_I\left(\int_Iu(x)^{1-p'}dx\right)^\frac{-1}{p}\left( \int_I \left[M(\mathbf{1}_Iu^{1-p'})(x)\right]^qw(x)dx\right)^\frac{1}{q}<\infty
\end{equation}
where the supremum is taken over all cubes $I \subset \R^n$. Replacing the weights $w,u$ with the measures $\o,\s$ we call $S^{p,q}_d(\o,\s)$ the dyadic Sawyer testing condition where the Maximal operator in \eqref{Sawyer testing} is replaced by the dyadic Maximal operator $M_d$ where the supremum in the operator is taken over only dyadic cubes.

\begin{thm}\label{A infinity and Sawyer testing}\label{Sawyer testing theorem}
($\s \in A_\infty$, $\A_p(\o,\s)\Rightarrow S^{p,p}_d(\o,\s)$) Let $\o,\s$ be Radon measures in $\R^n$ such that $\s \in A_\infty$. If $\o,\s$ satisfy the $\A_p(\o,\s)$ condition then the dyadic Sawyer testing condition $S^{p,p}_d(\o,\s)$ holds.
\end{thm}

\begin{proof}

Let $I$ be a cube in $\R^n$. Let $\Omega_m=\{x\in I:\left(M_d\mathbf{1}_{I}{\s}\right)(x)>K^m\}=\dot{\bigcup} I^m_j$, where $K$ is a constant to be determined later and $I^m_j$ are the maximal, disjoint dyadic cubes such that $\frac{\s(I^m_j)}{|I^m_j|}>K^m$. We have
\begin{eqnarray*}
\int_{I} \left(M_d\mathbf{1}_{I}{\s}\right)^{p}(x)d{\o}(x)
&\lesssim&\!\!\!\!
\sum_{m,j}\left(\frac{\s(I^m_j)}{|I^m_j|}\right)^p\o(I^m_j)\\
&=&\!\!\!\!
\sum_{m,j}\left(\frac{\s(I^m_j)^\frac{1}{p'}\o(I^m_j)^\frac{1}{p}}{|I^m_j|}\right)^p\!\!\!\s(I^m_j)
\leq
\A_p(\o,\s)\sum_{m,j}\s(I^m_j)
\end{eqnarray*}
Call $A^m_t=\bigcup_{I^{m+1}_j\subset I^m_t}I^{m+1}_j$. Since $\s \in A_\infty$ we get 
$$\s\left(A^m_t\right)\leq C\left(\frac{|A^m_t|}{|I^m_t|}\right)^\varepsilon\s(I^m_t)$$
for some $C$ positive and $\varepsilon$ like in \eqref{A infinity}. From the maximality of $I^m_j$ we obtain
$$
\left|A^m_t\right|=\!\!\!\sum_{I^{m+1}_j\subset I^m_t}\left|I^{m+1}_j\right|\leq \frac{1}{K^{m+1}}\s\left(A^m_t\right)\leq\frac{2^n}{K}|I^m_t|
$$
Choose $K$ big enough that $C\left(\frac{|A^m_t|}{|I^m_t|}\right)^\varepsilon\leq\frac{1}{2}$. 
Fix $m\in \N$, $k\geq -m$, then 
$$
\sum_j\s(I^k_j)\leq \left(\frac{1}{2}\right)^{\!\!\!m+k}\sum_j\s(I^{-m}_j)\leq \left(\frac{1}{2}\right)^{\!\!\!m+k}\!\!\!\!\!\!\s(I)
$$
$$
\sum_{k=-m}^\infty\sum_j\s(I^k_j)\leq \sum_{k=-m}^\infty 2^{-m-k}\s(I)\leq 2\s(I)
$$
and by taking $m \to \infty$ we get
$$
\sum_{k,j}\s(I^k_j)=\lim_{m \rightarrow \infty}\sum_{k=-m}^\infty\sum_j\s(I^k_j)\leq 2 \s(I)
$$
and this completes the proof of the theorem.
\end{proof}
With theorem \ref{Sawyer testing theorem}. at hand we get the following corollary.

\begin{cor}\label{A infinity and pivotal corollary}\label{pivotal by Ap}
($\o \in A_\infty$, $\A_p(\o,\s)\Rightarrow \V(\o,\s)^p$) Let $\o,\s$ be Radon measures in $\R^n$ such that $\s \in A_\infty$. Then the $\A_p(\o,\s)$ condition implies the pivotal condition $\V(\o,\s)^p$.
\end{cor}
\begin{proof}
Let $I$ be a cube in $\R^n$.
\begin{eqnarray}\label{poisson by maximal bound}
\mathrm{P}(I,{\s})
&=&
\!\!\!\!\int\frac{|I|}{\left(|I|^{\frac{1}{n}}+|x-x_I|\right)^{2n}}d{\s}(x)
\!\lesssim
\sum_{m=0}^\infty \frac{{\s}\big((2^m+1)I\big)}{2^m|2^mI|}\\
\!\!\!\!&\lesssim&\!\!\!\!
\sum_{m=0}^\infty \inf_{x \in I}M_d{\s}(x)2^{-m}
\lesssim
\inf_{x \in I}M_d{\s}(x)\notag
\end{eqnarray}
where $M_d$ denotes the dyadic maximal function.

Let $I_0$ be a cube in $\R^n$. Let $I_0=\bigcup_{r\geq 1}I_r$ be a decomposition of $I_0$ in disjoint cubes. Using \eqref{poisson by maximal bound} we get
$$
\sum_{r \geq 1}{\o}(I_r){\mathrm{P}}^p(I_r,\mathbf{1}_{I_0}{\s})
\leq 
\sum_{r\geq 1}{\o}(I_r) \inf_{x \in I_r}\left(M_d\mathbf{1}_{I_0}{\s}\right)^p(x)
\leq 
\int_{I_0} \left(M_d\mathbf{1}_{I_0}{\s}\right)^p(x)d{\o}(x)
$$
and using theorem \ref{Sawyer testing theorem}. the last expression is bounded by a constant multiple of $\s(I_0)$. So we have 
$$
\sum_{r \geq 1}{\o}(I_r){\mathrm{P}}^p(I_r,\mathbf{1}_{I_0}{\s})
\leq K\s(I_0)
$$
and that completes the proof of the corollary.
\end{proof}

\begin{question}\label{doubling measures question}
$A_\infty$ is a special class of doubling measures. Is it true that for $\o$ doubling measure, $\A_p(\o,\s)\Rightarrow \V(\o,\s)^p$?
\end{question}
\begin{question}
In corollary \ref{A infinity and pivotal corollary} we prove that for $\o\in A_\infty$ dyadic Sawyer testing implies pivotal. Is it true that $S_d^{p,p}(\o,\s)=\V(\o,\s)^p$? 
\end{question}
\begin{rem}
The proof of corollary \ref{pivotal by Ap}, holds also for the fractional $\A_p(\o,\s)$ and pivotal conditions as stated in \cite{SSU4} (stated for $p=2$ but extends immediately to any $p>1$). 
\end{rem} 

\textit{Proof of theorem \ref{T1 theorem}}:
If both the measures $\o, \s$ are in the one weight $A_\infty$, then by corollary \ref{A infinity and pivotal corollary}, the the two weight $\A_2(\o,\s)$ condition implies both the pivotal conditions $\V(\o,\s)^2$ (\eqref{pivotal} and it's dual) and we can apply the main theorem from \cite{SSU4} (or the one in \cite{LW}) to get the result. \qed

\subsection{The ``buffer" conditions do not imply the tailed $\A_p$ conditions.}

The goal of this subsection is to give a proof of theorem \ref{non doubling pivotal example}. First we make the following simple remark.

\begin{rem}\label{pivotal implies classical Ap}
It is immediate to see that the Pivotal condition implies the classical $\A_p$ condition. Just let the decomposition in \eqref{pivotal} be just a single cube.
\end{rem}

\begin{rem}
We are going to use $p=2$ in the proof for simplicity. The proof works for $1<p\leq 2$, without any modifications.
\end{rem}

\textit{Proof of theorem \ref{non doubling pivotal example}:} We construct measures $\o$ and $\s$ so that the pivotal condition  $\V^2$ \eqref{pivotal} holds, but $\A_2^{t_1}$ (\ref{one tailed}) does not. Let 
$$
\displaystyle\o(E)=\delta_0,\quad \s(E)=\sum_{n=2}^\infty n\delta_n(E)
$$
where $\delta_n$ denotes the point mass at $x=n$. First we check $\A_2^{t_1}$ does not hold. Let $I=[0,1]$. Then

$$\frac{\o(I)}{|I|}P(I,\s)=\int_\R \frac{1}{(1+dist(x,[0,1]))^2}\s (dx)=\sum_{n=2}^\infty \frac{n}{(n)^2}=\infty$$

To show the pivotal condition holds, let $I_0=(a,b)$ where $a<0$ and $n \leq b <n+1, n\geq 2$ (we need $I_0$ to contain some masses from $\s$ and $0 \in I_0$ for otherwise there is nothing to prove). Decomposing $I_0=\dot{\cup}I_r$, only the $I_r$ such that $0 \in I_r$ contributes to the pivotal condition. Call that cube $I_1$. We consider the cases:

\begin{enumerate} [(i)]
    \item $|I_1| \leq 1$. We calculate:
    
\begin{eqnarray*}
    \frac{\o(I_1) P(I_1,I_0\s)^2}{\s(I_0)}=\frac{\displaystyle\left(\int_{I_0} \frac{|I_1|\s (dx)}{(|I_1|+\dist(x,I_1)))^2}\right)^2}{\s(I_0)}
    \!\!\!&\leq& \!\!\!
    \displaystyle|I_1|^2\left(\sum_{k=2}^{n}\frac{k}{k^2}\right)^2\Bigg/\displaystyle\sum_{k=2}^{n}k\\
    \!\!\!&\leq&\!\!\!
    M<\infty
\end{eqnarray*}
    where the constant $M$ does not depend on $n$.
    
    \item $|I_1| \geq n$. We get:

\begin{eqnarray*}
    \frac{\o(I_1) P(I_1,I_0\s)^2}{\s(I_0)}=\frac{\displaystyle\left(\int_{I_0} \frac{|I_1|\s (dx)}{(|I_1|+\dist(x,I_1)))^2}\right)^2}{\s(I_0)}
    \!\!\!&\leq&\!\!\!
    \displaystyle\displaystyle|I_1|^2\left(\sum_{k=2}^{n}\frac{k}{|I_1|^2}\right)^2
    \!\!\! \Bigg/\displaystyle\sum_{k=2}^{n}k\\
    \!\!\!&\leq& \!\!\!
    \frac{n^2}{|I_1|^2}\leq 1
\end{eqnarray*}
    
    \item $1\leq |I_1|\leq n$. We have:
    
    \begin{eqnarray*}
    \frac{\o(I_1) P(I_1,I_0\s)^2}{\s(I_0)}\lesssim \frac{|I_1|^2}{{n}^2}\Bigg(\sum_{k=2}^{|I_1|}\frac{k}{|I_1|^2}+\sum_{k=|I_1|}^{n}\frac{1}{k}\Bigg)^2
    &\lesssim&
    \frac{|I_1|^2}{{n}^2}\bigg(1+\log\Big(\frac{n}{|I_1|}\Big)\bigg)^2\\
    &\lesssim&
    \frac{|I_1|^2}{{n}^2}+\frac{|I_1|^2}{{n}^2}\log^2\Big(\frac{n}{|I_1|}\Big)
    \end{eqnarray*}
\end{enumerate}
Now, on the last expression setting $x=\frac{{n}}{|I_1|}$ we get the function $f(x)=\frac{\log^2x}{x^2}$, $x\geq 1$ which is bounded independent of $n$.
Combining all three cases we see that the pivotal condition is bounded.\qed 

\begin{question}
In the above example, one can check that the dual pivotal condition does not hold. Is it true that $\V(\s,\o)^p\cap \V(\o,\s)^p \Rightarrow \A_p^{t_1}(\o,\s)$?
\end{question}

\subsubsection{Doubling measures and the Pivotal condition.}The result in this subsection is essentially in \cite{Saw2}, equation (4.4), but we include it for completeness. We partially answer positively \textbf{question \ref{doubling measures question}}.

If the measures $\o,\s$ are doubling but not in $\A_\infty$ then we do not in general know if the Pivotal condition can be controlled by the $\A_p(\o,\s)$ condition. For measures with small doubling constant though the $\A_p(\o,\s)$ condition implies $\V(\o,\s)^p$.

\begin{thm}\label{Ap and small doubling corollary}
(Small doubling+$\A_p(\o,\s)\Rightarrow \V(\o,\s)^p$) Let $\o, \s$ be doubling measures in $\R^n$ with doubling constants $K_\o,K_\s$ and reverse doubling constants $1+\delta_\o,1+\delta_\s$ respectively. If $K_\s<2^p(1+\delta_\o)$ then the $\A_p(\o,\s)$ condition implies the pivotal condition $\V(\o,\s)^p$.
\end{thm}

\begin{proof}
Let $I_0$ be a cube in $\R^n$ and $I_0=\cup_{r \geq 1}I_r$ be a decomposition of $I_0$ in disjoint cubes.

\begin{eqnarray*}
&&\sum_{r\geq 1}\o(I_r)P(I_r,I_0\s)^p
\approx
\sum_{r \geq 1}\o(I_r)\left(\sum_{m=1}^{m_r}\frac{\s(I^m_r)}{2^m|I^m_r|}\right)^p\\
&\leq&
\sum_{r\geq 1}\left(\sum_{m=1}^{m_r}(1+\delta_\o)^{-\frac{m}{p}}\frac{\o^\frac{1}{p}(I^m_r)\s^\frac{1}{p'}(I^m_r)}{2^m|I^m_r|}\s^\frac{1}{p}(I^m_r)\right)^p\\
&\leq&
\A_p(\o,\s)\sum_{r \geq 1}\s(I_r)\left(\sum_{m=1}^{m_r}\left(\frac{K_\s}{2^p(1+\delta_\o)}\right)^\frac{m}{p}\right)^p\lesssim \A_p(\o,\s)\s(I_0)
\end{eqnarray*}
where $m_r=\log_2 \left(\frac{|I_0|}{|I_r|}\right)^\frac{1}{n}$, $I^m_r$ is the cube with same center as that of $I_r$ and $|I^m_r|^\frac{1}{n}=2^m|I_r|^\frac{1}{n}$. where the implied constant depends only on the doubling constant of $\s$ and the reverse doubling constant of $\o$. This completes the proof of the theorem.
\end{proof}
\begin{rem}
For a doubling measure $\o$ and a cube $I\subset\R^n$ we have that in \eqref{energy gain} $E(I,\o)^2\geq c_\o>0$ since $\o(I_1)\approx \o(I_{2})$ where $|I_1|=|I_{2}|=2^{-n}|I|$ and $I_1$ is in the top left corner of $I$, $I_{2}$ in the bottom right corner of $I$. Hence for $\o$ doubling the Pivotal condition $\V(\o,\s)^p$ is equivalent to the Energy condition $\E(\o,\s)^p$.
\end{rem}

\end{document}